\documentclass[11pt]{amsart}
  \usepackage{amsmath, amsfonts, amssymb,amsthm}

\newtheorem{theorem}{Theorem}

\newtheorem{cor}{Corollary}

\theoremstyle{definition}

\def\min{\mathop{\mathrm{min}}}

\newcommand{\C}{\mathbb{C}}

\newcommand{\PP}{\mathbb{P}}

\begin{document}
\title{{ A second main theorem for holomorphic curve intersecting hypersurfaces}}

\author{Nguyen Van Thin}
\address{Department of Mathematics, Thai Nguyen University of Education, Luong Ngoc Quyen Street, Thai Nguyen city, Viet Nam.}
\email{thinmath@gmail.com}

\thanks{2000 {\it Mathematics Subject Classification.} Primary 32H30.}
\thanks{Key words: Nevanlinna theory, Holomorphic curve.}

\begin{abstract}
In this paper, we establish a second main theorem for holomorphic curve intersecting hypersurfaces in general position
in projective space with level of truncation. As an application, we reduce the number hypersurfaces
 in uniqueness problem for holomorphic curve of authors before. 
\end{abstract}
\baselineskip=16truept 
\maketitle 
\pagestyle{myheadings}
\markboth{}{}
\section{Introduction and main results}

\def\theequation{1.\arabic{equation}}
\setcounter{equation}{0}

Let $f$ be a holomorphic curve of $\C$ into $\mathbb P^{n}(\C).$ For arbitrary fixed homogeneous coordinates 
$(w_0: \dots : w_n)$ of $\mathbb P^{n}(\C),$ we take a reduced representation $f=(f_0:\dots : f_n)$ which means that each $f_i$
is a holomorphic function on $\C$ without common zeros. Set 
$||f(z)||=\max\{|f_0(z)|, \dots, |f_n(z)|\}.$
The characteristic function of $f$ is defined by
$$ T_f(r)=\dfrac{1}{2\pi}\int_{0}^{2\pi}\log||f(re^{i\theta})||d\theta,$$
where the above definition is independent, up to an additive constant, of the choice of the reduced representation of $f.$
 
Let $D$ be a hypersurface in $\PP^n(\C)$ of degree $d$. Let $Q$ be the homogeneous polynomial of degree $d$ defining $D$. 
Under the assumption that $ Q(f) \not\equiv 0,$ the proximity function $m_f(r, D)$ of $f$ is defined by
$$m_{f}(r,D)=\dfrac{1}{2\pi}\int_{0}^{2\pi} \log \dfrac {\|f(re^{i\theta})\|^{d}}{|Q(f)(re^{i\theta})|} d\theta,$$
where the above definition is independent, up to an additive constant, of the choice of the reduced representation of $f$. 
To define the counting function, let $n_f(r, D)$ be the number of zeros of $Q(f)$ in the disk $|z | < r,$ counting multiplicity. The counting function $N_f(r, D)$ is then defined by
$$ N_f(r, D)=\int\limits_{0}^{r}\dfrac{n_f(t, D)-n_f(0, D)}{t}dt+n_f(0, D)\log r. $$
For a positive integer $M$, we denote by $n_f^{M}(r, D)$ the number of zeros of $Q(f)$ inside $|z | < r,$ counting multiplicities for those orders $\le  M,$ and being counted as $M$ if the order is $> M.$ The truncated counting function
to level $M$ is defined by
$$ N_f^{M}(r, D)=\int\limits_{0}^{r}\dfrac{n_f^{M}(t, D)-n_f^{M}(0, D)}{t}dt+n_f^{M}(0, D)\log r. $$

The Poisson-Jensen formula implies the First Main Theorem.
\begin{theorem} \cite{Ru1} Let $f: \mathbb C \to \mathbb P^N(\mathbb C)$ be a holomorphic
curve, and let $D$ be a hypersurface in $\mathbb P^N(\mathbb C)$ of degree $d$. If $f(\mathbb C)\not\subset D,$ then
for every real number $r$ with $0 < r < +\infty,$
$$ dT_f(r)=m_f(r, D)+N_f(r, D) +O(1),$$
where $O(1)$ is a constant independent of $r.$
\end{theorem}

Let $f:\mathbb C\to \mathbb P^{N}(\mathbb C)$ be a holomorphic curve. The map $f$ is said to be linearly 
non-degenerate if the image of $f$ is not contained in any linear proper subspace of $\mathbb P^{N}(\mathbb C).$ Under this
 assumption, in 1933, H. Cartan proved the following Second Main Theorem.
\begin{theorem}\cite{CA}\label{th2} Let $f: \mathbb C \to \mathbb P^N(\mathbb C)$ be a linearly nondegenerate holomorphic
  curve. Let $H_1, \dots, H_q$ be hyperplanes in $\mathbb P^N(\mathbb C)$ in general position. Then,
$$\| (q -(N + 1))T_f(r) \le \sum_{j=1}^{q}N_f^{N}(r, H_j) + o(T_f(r)),$$
here $''\|''$ means that the inequality holds for all $r\in [0, +\infty)$ except a set of finite Lebesgue measure.
\end{theorem}
M. Ru \cite{Ru1} extended the above Cartan's result to an algebraically non-degenerate holomorphic curve 
$f:\mathbb C\to \mathbb P^{N}(\mathbb C)$ intersecting hypersurfaces in 2004. After that, Q. M. Yan and Z. H. Chen \cite{YC}
 show that there exists the multiplicity truncation bounded in the second main theorem of M. Ru \cite{Ru1}. When one applies inequalities of second main theorem type, it is often
crucial to the application to have the inequality with truncated counting functions. For example, all existing constructions of unique range sets depend on a second
main theorem with truncated counting functions. T. T. H. An and H. T. Phuong \cite{AP} gave an explicit estimate for the level of truncation
 in Ru's result. However, the level of truncation is large and is depending on $\varepsilon.$

Recall that hypersurfaces $D_1,\dots, D_q , q > N,$ in $\mathbb P^N(\mathbb C)$ are said to be in general
position if for any distinct $i_1,\dots, i_{N+1}\subset \{1, \dots, q\},$
$$ \cap_{k=1}^{N+1}\text{supp} D_{i_k}=\varnothing.$$

Our main results is as follows:

\begin{theorem}\label{th3}
Let $D_j, j=1,\dots, q,$ be hypersurfaces with degree $d_j$ in general position in $\mathbb P^{N}(\mathbb C).$ Suppose that
 $Q_1, \dots, Q_q$ are homogeneous polynomials with degree $d_j$ defining $D_1, \dots, D_q,$ respectively. We assume that 
$n=lcm(d_1, \dots, d_{q})>1$ and  $\sum_{j=1}^{q}Q_j^{n/d_j}\not\equiv 0.$ Denote by 
$D_{q+1}=\{z\in \mathbb P^N(\mathbb C): \sum_{j=1}^{q}Q_j^{n/d_j}=0\}.$ 
Suppose that $n=lcm(d_1, \dots, d_q)>1.$ Let $f: \mathbb C \to \mathbb P^N(\mathbb C)$ be a
 algebraically nondegenerate holomorphic curve. Then,
$$\| T_f(r) \le \sum_{j=1}^{q}\dfrac{1}{d_j}N_f^{q-1}(r, D_j)+\dfrac{1}{n}N_f^{q-1}(r, D_{q+1}) + o(T_f(r)).$$
\end{theorem}
From Theorem \ref{th3}, we get the following result:
\begin{cor}\label{corth3}
Let $D_j, j=1,\dots, N+1,$ be hypersurfaces with degree $d_j$ in general position in $\mathbb P^{N}(\mathbb C).$ Suppose that
 $Q_1, \dots, Q_{N+1}$ are homogeneous polynomials with degree $d_j$ defining $D_1, \dots, D_{N+1},$ respectively. 
We assume that $n=lcm(d_1, \dots, d_{N+1})>1$ and  $\sum_{j=1}^{N+1}Q_j^{n/d_j}\not\equiv 0.$ Denote by 
$D_{N+2}=\{z\in \mathbb P^N(\mathbb C): \sum_{j=1}^{N+1}Q_j^{n/d_j}=0\}.$ 
Let $f: \mathbb C \to \mathbb P^N(\mathbb C)$ be a
 algebraically nondegenerate holomorphic curve. Then,
$$\| T_f(r) \le \sum_{j=1}^{N+1}\dfrac{1}{d_j}N_f^{N}(r, D_j)+\dfrac{1}{n}N_f^{N}(r, D_{N+2}) + o(T_f(r)).$$
\end{cor}
Note that the truncated level $N$ of the counting function in Corollary \ref{corth3} is much smaller than the previous results
  of all other authors \cite{AP, S} and is independent of $\varepsilon.$ From  Theorem \ref{th3}, we have the result about the
 algebraically degeneracy of holomorphic map from $\mathbb C$ to $\mathbb P^N(\mathbb C)$ as follows.
\begin{theorem}\label{th4}
Let $D_j, j=1,\dots, q+1,$ be hypersurfaces with degree $n$ as in Theorem \ref{th3}. Let 
$f: \mathbb C \to \mathbb P^N(\mathbb C)$ be a holomorphic curve such that image of $f$ intersect $D_j, j=1, \dots, q+1,$ with multiplicity at least $l_j,$ respectively.
Suppose that $\sum_{j=1}^{q+1}\dfrac{1}{l_j}<\dfrac{1}{q-1},$ then $f$ is a algebraically degenerate holomorphic curve.
\end{theorem}
Theorem \ref{th4} is the first result for algebraically degenerate of holomorphic curve intersecting $N+2$ hypersurfaces
with suitable multiplicity (we choose $q=N+1$).

Now, we give a application of Theorem \ref{th3} in uniqueness problem for holomorphic curve sharing hypersurfaces.
\begin{theorem}\label{th5}
Let $D_j, j=1,\dots, q+1,$ be hypersurfaces with degree $n$ as in Theorem \ref{th3}. Let
 $f, g : \mathbb C \to \mathbb P^N(\mathbb C)$ be two algebraically nondegenerate
holomorphic curves, and $n$ be a integer with $n>2(q^2-1).$ Suppose that $f(z)=g(z)$ on $\cup_{j=1}^{q+1}(f^{-1}(D_j)
\cup g^{-1}(D_j)),$ then $f\equiv g.$
\end{theorem}
 We reduce the number hypersurfaces in before results. The authors \cite{DR, P} studied the
 uniqueness problem with a number lager hypersurfaces. In Theorem \ref{th5}, if we take $q=N+1,$ then we only need $N+2$ hypersurfaces in 
$\mathbb P^N(\mathbb C)$ for uniqueness problem of holomorphic curve.

\section{Proof of theorems}

\begin{proof}[Proof of Theorem \ref{th3}]
First, we suppose that $Q_1, \dots, Q_q$ have the same degree $n.$ Given $z\in \mathbb C$ there exists a renumbering 
$\{i_0,\dots, i_{N}\}$ of the indices $\{1,\dots , q \}$
such that
\begin{align}\label{cta1}
0<|Q_{i_0}\circ (f (z ))| \le  |Q_{i_2}\circ (f(z))| \le  \dots \le |Q_{i_{N}}\circ (f(z))|\le \min_{l\not \in\{i_0, \dots, i_N\}}
|Q_{l}\circ (f(z))|.
\end{align}
 From the hypothesis, $D_1, \dots, D_q$ are hypersurfaces  which are located in general position  in $\PP^{N}(\C),$
 we have for every subset $\{i_0, \dots, i_N\} \subset \{1, \dots, q\},$
$$ \text{Supp}D_{i_{0}} \cap \dots \cap \text{Supp}D_{i_N}= \emptyset.$$
Thus by Hilbert's Nullstellensatz \cite{VW}, for
any integer k, $0 \le k \le N,$ there is an integer $m_k > n$ such that
$$ x_k^{m_k}=\sum_{j=0}^{N}b_{k_j}(x_0, \dots, x_N) Q_{i_j}(x_0, \dots, x_N),$$
where $b_{k_j}$ are homogeneous forms with coefficients in $\mathbb C$ of degree $m_k-n.$
This implies
$$ 
|f_k(z )|^{m_k}\le c_1||f (z )||^{m_k-n}\max\{|Q_{i_0}\circ (f (z ))|, \dots, |Q_{i_N}\circ (f (z ))|\},$$
where $c_1$ is a positive constant depending only on the coefficients of $b_{k_j}, 0\le j\le N, 0\le k\le N,$ thus depends only
 on the coefficients of $Q_l, 1\le l\le q.$ Therefore,
\begin{align}\label{cta2}
||f(z )||^{n}\le c_1\max\{|Q_{i_0}\circ (f (z ))|, \dots, |Q_{i_N}\circ (f (z ))|\}.
\end{align}
Since $D_j, j=1, \dots, q,$ are hypersurfaces in general position, then  $F=(Q_1\circ(f(z)), \dots, Q_q\circ(f(z)))$ is a 
holomorphic curve from $\mathbb C$ into $\mathbb P^{q-1}(\mathbb C)$ with reduced form
 $F=(Q_1\circ(f(z)):\dots:Q_q\circ(f(z))).$ From (\ref{cta2}), we have
$$ ||f(z )||^{n}\le c_1\max\{|Q_{1}\circ (f (z ))|, \dots, |Q_{q}\circ (f (z ))|\}.$$
This implies 
\begin{align}\label{41}
T_F(r)=nT_f(r)+O(1).
\end{align} 
Applying Theorem \ref{th2} for $F$, with the hyperplanes which locate general 
position in $\mathbb P^{q-1}(\mathbb C)$
$$ H_i=\{y_i=0\}, 0\le i\le q-1,$$
and
$$ H_{q}=\{y_0+\dots+y_{q-1}=0\} $$
yields that the inequality
\begin{align}\label{42}
\|T_F(r)\le \sum_{i=0}^{q}N_F^{q-1}(r, H_i)+o(T_f(r)).
\end{align}
We have $N_F^{q-1}(r, H_i)= N_f^{q-1}(r, D_{i+1})$ for all $i=0,\dots, q-1,$ and  
$N_F^{q-1}(r, H_{q})=N_f^{q-1}(r, D_{q+1}).$ 
Therefore from (\ref{41}) and (\ref{42}), we obtain 
\begin{align*}
\| nT_f(r)\le \sum_{j=1}^{q+1}N_f^{q-1}(r, D_j)+o(T_f(r)).
\end{align*}
If $D_1, \dots, D_q$ have not the same degree, we consider the hypersurfaces $D_1^{n/d_1}, \dots, D_q^{n/d_q},$
 where $n=lcm(d_1, \dots, d_q).$ Then $D_1^{n/d_1}, \dots, D_q^{n/d_q}$ have the same degree $n.$ Apply to above
 result, we get
\begin{align*}
\| nT_f(r)\le \sum_{j=1}^{q}N_f^{q-1}(r, D_j^{n/d_j})+N_f^{q-1}(r, D_{q+1})+o(T_f(r)).
\end{align*}
Note that $N_f^{q-1}(r, D_j^{n/d_j})=\dfrac{n}{d_j}N_f^{q-1}(r, D_j)$ for all $j=1, \dots, q.$ Then we obtain
\begin{align*}
\| T_f(r)\le \sum_{j=1}^{q}\dfrac{1}{d_j}N_f^{q-1}(r, D_j^{n/d_j})+\dfrac{1}{n}N_f^{q-1}(r, D_{q+1})+o(T_f(r)).
\end{align*}
\end{proof}

\begin{proof}[Proof of Theorem \ref{th4}]
Suppose that $f$ is a algebraically nondegenerate holomorphic curve. From Theorem \ref{th3} and the First Main Theorem, we have
\begin{align*}
\| nT_f(r)&\le \sum_{j=1}^{q+1}N_f^{q-1}(r, D_j)+o(T_f(r))\\
&\le (q-1)\sum_{j=1}^{q+1}N_f^1(r, D_j)+o(T_f(r)).
\end{align*}
Then, we have
\begin{align*}
\| nT_f(r)&\le (q-1)\sum_{j=1}^{q+1}\dfrac{1}{l_j}N_f(r, D_j)+o(T_f(r))\\
&\le (q-1)\sum_{j=1}^{q+1}\dfrac{n}{l_j}T_f(r)+o(T_f(r)).
\end{align*}
This implies 
$$\|\Big(\dfrac{1}{q-1}-\sum_{j=1}^{q+1}\dfrac{1}{l_j}\Big)T_f(r)\le o(T_f(r)),$$
which gives a contradiction with $\sum_{j=1}^{q+1}\dfrac{1}{l_j}<\dfrac{1}{q}.$ Therefore, $f$ must be an algebraically 
degenerate.
\end{proof}

\begin{proof}[Proof of Theorem \ref{th5}]
We suppose that $f\not\equiv g,$ then there are two numbers $\alpha, \beta \in  \{0, \dots , N\},$ $ \alpha\ne \beta$ such that
$f_{\alpha}g_{\beta}\not\equiv f_{\beta}g_{\alpha}.$ Assume that $z_0 \in \cup_{j=1}^{q+1}(f^{-1}(D_j)
\cup g^{-1}(D_j)),$ from condition 
$f(z)=g(z)$ when $z\in \cup_{j=1}^{q+1}(f^{-1}(D_j)
\cup g^{-1}(D_j)),$ we get $f(z_0)=g(z_0).$ This implies $z_0$ is a zero of $\dfrac{f_{\alpha}}{f_{\beta}}-\dfrac{g_{\alpha}}{g_{\beta}}.$ Therefore, we have
\begin{align*}
 N_f^{q-1}(r, D_j) \le (q-1)N_f^{1}(r, D_j)&\le (q-1)N_{ \dfrac{f_{\alpha}}{f_{\beta}}-\dfrac{g_{\alpha}}{g_{\beta}}}(r,0 )\\
&\le (q-1)(T_f(r)+T_g(r))+O(1)
\end{align*}
for all $j=1, \dots, q+1.$
Apply to Theorem \ref{th3}, we obtain
\begin{align}\label{51}
  \| nT_f(r)\le (q^2-1)(T_f(r)+T_g(r))+o(T_f(r)).
\end{align}
Similarly, we have
\begin{align}\label{52}
  \| nT_g(r)\le (q^2-1)(T_f(r)+T_g(r))+o(T_g(r)).
\end{align}
Combining (\ref{51}) and (\ref{52}), we get
$$ \|(n-2(q^2-1))(T_f(r)+T_g(r)) \le o(T_f(r))+o(T_g(r)).$$
This is a contradiction with $n>2(q^2-1).$ Hence $f\equiv g.$
\end{proof}

\end{document}